\documentclass[11pt,a4paper]{article}
\usepackage{latexsym,amsmath,amssymb,graphicx,slashbox}
\usepackage{color,url}

\def\reff#1{(\ref{#1})}
\def\T{^\top}
\def\eps{\varepsilon}
\def\R{{\mathbb R}}

\newtheorem{thm}{Theorem}[section]
\newtheorem{lemma}{Lemma}[section]
\newtheorem{corollary}{Corollary}[section]
\newtheorem{remark}{Remark}[section]
\newtheorem{proposition}{Proposition}[section]

\newenvironment{proof}{\noindent {\bf Proof. }}{\hfill $\Box$ \\ \medskip}
\def\qf#1#2{#1\T  #2 #1}

\def\bd{\mbox{\rm bd}\,}

\def\diag{\mbox{\rm diag}\,}

\def\conv{\mbox{\rm conv\,}}
\def\tr{\mbox{\rm trace}\,}
\def\rank{\mbox{\rm rank}\,}
\def\ext{\mbox{\rm ext}\,}
\def\beq#1{\begin{equation}\label{#1}}
\def\eeq{\end{equation}}
\def\bep{\begin{proof}}
\def\ep{\end{proof}}

\def\lk{\left\{}
\def\rk{\right \} }

\def\norm#1{ \| #1 \|}

\def\cs#1{{{\cal C}_{#1}^*}}
\def\cc#1{{\cal C}_{#1}}

\def\ca{{\cal A}}
\def\ckk{{\cal K}}
\def\csn#1{{{\cal S}_{#1}}}
\def\csnp{\PP_n}
\def\x{\mathbf x}
\def\y{\mathbf y}
\def\z{\mathbf z}
\def\u{\mathbf u}
\def\v{\mathbf v}

\def\w{\mathbf w}

\def\o{\mathbf o}
\def\e{\mathbf e}

\def\PP{{\mathcal P}}
\def\NN{{\mathcal N}}
\def\LL{{\mathcal L}}

\def\FF{{\mathcal F}}
\def\BB{{\mathcal B}}

\def\inn#1{{\rm int}\,{#1}}

\def\supp#1{\sigma (#1)}
\def\cpr{\mbox{\rm cpr}\,}
\def\ignore#1{}
\def\corr#1{#1}
\def\cor#1{#1}
\def\cob#1{#1}
\definecolor{purple}{rgb}{0.5,0,0.5}
\def\cop#1{#1}
\definecolor{dgreen}{rgb}{0,0.7,0}
\def\cogg#1{#1}
\def\cog#1{#1}
\def\cobf#1{#1}
\def\copa#1{#1}

\def\bea#1{\begin{array}{#1}}
\def\ea{\end{array}}
\begin{document}
\begin{titlepage}
{\Large \raggedright
New results on the cp rank and\\[1em]
related properties of co(mpletely~)positive matrices}

\bigskip
\begin{flushleft}
\normalsize{
{\bf Naomi Shaked-Monderer}\footnotemark \hfill{Emek Yezreel College, 
Israel}\\ [2ex]
{\bf Abraham Berman}\footnotemark[1] \hfill {Dept.~of Mathematics, Technion, Haifa, Israel}\\ [2ex]
{\bf Immanuel M. Bomze} \hfill {ISOR, University of Vienna, Austria}\\ [2ex]
{\bf Florian Jarre} \hfill {Mathem.~Institut, University of D\"usseldorf, Germany} \\ [2ex]
{\bf Werner Schachinger} \hfill {ISOR, University of Vienna, Austria}
}

\end{flushleft}
\rule{\textwidth}{0.75pt} \vskip0.4cm
\begin{abstract}
\normalsize
Copositive  and completely positive  matrices play an increasingly important role in Applied Mathematics,
namely as a key concept for approximating NP-hard optimization problems. The cone  of copositive matrices
of a given order  and the cone of completely positive  matrices  of  the same order  are dual to each other
with respect to  the standard scalar product on the space of symmetric matrices. This paper establishes some
new relations between orthogonal pairs of such matrices lying on the boundary of either cone. As a consequence,
we can establish an improvement on the upper bound of the cp-rank of completely positive matrices of general
order, and a further improvement for such matrices of order six.
\end{abstract}
\bigskip

\noindent
{\bf Key words:}
copositive optimization, completely positive matrices, cp-rank, nonnegative factorization.

\bigskip
\noindent
AMS classification: 15B48, 90C25, 15A23

\bigskip
\bigskip

\centerline{\today}

\footnotetext[1]{The work of Abraham Berman and  Naomi Shaked-Monderer was supported by grant no.\ G-18-304.2/2011 by the German-Israeli Foundation for Scientific Research and Development (GIF).}

\end{titlepage}
\newpage
\setcounter{page}{1}

\section{Introduction}

In this article we consider completely positive matrices and their  cp-rank, as well as copositive matrices.
An $n\times n$ matrix $M$ is  said to be {\em completely positive} if there exists a nonnegative  (not necessarily square) matrix $V$ such that $M=VV\T$. An $n\times n$ matrix $A$ is  said to be {\em copositive} if $\qf{\x} A  \ge 0$ for every nonnegative vector  $\x\in \R^n_+$.
The completely positive matrices of order $n$ form a cone, $\cs n$, dual to the cone of copositive matrices of that order, $\cc n$. Both cones are central in the rapidly evolving field of {\em copositive optimization} which links discrete and continuous optimization, and has
numerous real-world applications. For recent surveys and structured bibliographies, we refer to~\cite{Bomz11a,Bomz11b,Bure11,Duer10}, and for a fundamental text book to~\cite{Berm03}.

A main motivation for this paper was the study of cp-rank:
A given completely positive matrix $M\!\!\ne\!0$ always has many factorizations~$M\!\!=\!\!VV\T$, where $V$ is a nonnegative matrix, and
the {\em cp-rank} of $M$, $\cpr M$, is the minimum number of columns in such a nonnegative factor $V$  (for completeness, we define $\cpr M = 0$ if $M$ is a square zero matrix and $\cpr M = \infty$ if $M$ is not completely positive).
Determining the maximum possible cp-rank of $n\times n$ completely positive matrices,
$$p_n := \max\lk \cpr M : M \mbox{ is a completely positive }n\times n\mbox{ matrix}\rk ,$$
is still an open problem for  large $n$
 (up to now, for $n\ge 6$; only recently $p_5=6$ has been established~\cite{Shak13}).
It is known \cite[Theorem~3.3]{Berm03} that
\beq{pnsmall}p_n=n\quad\mbox{if }n\le 4 \, .\eeq
For $n\in \lk 2,3\rk$, there exist simple proofs of~(\ref{pnsmall}), but already for $n=4$, the argument is quite involved~\cite{Berm03}.
For $n\ge 5$, it is known that
\beq{pnbracket} {\cob{ d_n :=}}\left\lfloor \frac{n^2}4\right\rfloor \le p_n \le {n+1\choose 2}-1\, ,\eeq
but whether the lower bound {\cob{$d_n$}} is in fact equal to $p_n$ is still unknown for large $n$. This is the famous Drew-Johnson-Loewy (DJL) conjecture~\cite{Drew94}. The above upper bound on $p_n$ on the right-hand side
\cobf{follows, for example,} 
from the so-called Barioli-Berman~\cite{Bari03} bound:
Let
$$b_r :=\max\lk \cpr M : M \mbox{ is a completely positive  matrix with }\rank M = r\rk ,$$
then for $r\ge 3$
\beq{babe} b_r  = {r+1\choose 2}-1 \, .\eeq
{Some evidence in support of the DJL conjecture is found in
\cite{Drew94,Drew96a,Berm98,Loew03}, see also \cite[Section 3.3]{Berm03}.}
The DJL conjecture has recently been proved for $n=5$ \cite{Shak13}, but
the cp-rank problem is still not fully resolved. Not only is it  not known whether the DJL conjecture holds, but the best upper bound on $p_n$ for $n\ge 6$ remained, for over a decade, $b_n$. Two main results of this paper are a reduction of the upper bound on $p_n$ in  the bracket~\reff{pnbracket} for general $n$ and a further reduction in case of $n=6$. To obtain  these results, we use
 \cite[Thm.\,3.4]{Shak13}, which guarantees that $p_n$ is attained (also) at a nonsingular matrix on the boundary of the cone of $\cs n$. We also complement this result here by studying,  for every possible cp-rank $1\le k\le p_n$, where in $\cs n$ the cp-rank $k$ is attained.

 Each matrix on the boundary of the cone $\cs n$ is orthogonal to a matrix on the boundary of the cone $\cc n$ (in fact, to a matrix generating an extreme ray of that cone). Thus to improve the bound on $p_n$ we consider pairs of matrices,
 $M\in \cs n$ and $A\in \cc n$, that are orthogonal to each other  in the standard scalar product of matrices. This leads also to some results that are not directly related to the cp-rank problem, and are of interest in their own right.

The paper is organized as follows: after introducing basic concepts and terminology, we show, in Section~\ref{orthdiagdom}, some important orthogonality and diagonal dominance results. Section~\ref{extlow} is devoted to the study of extreme copositive matrices of low rank, while Section~\ref{genercpr} considers
completely positive matrices  having a fixed cp-rank: where in the completely positive cone they can be found, and whether they form a substantial part of the cone. Section~\ref{newcprbounds} presents improvements of upper bounds on the cp-rank for matrices of general order, and a further tightening of this bound for order six is put forward in the final Section~\ref{psix}.

 Some notation and terminology:
let $\e_i$ be  the $i$th column vector of the $n\times n$ identity matrix $I_n$.
The nonnegative orthant is denoted by $\R^n_+$. For a vector $\x\in \R^n_+$, the {\em support}  of $\x$ is denoted by $\sigma( \x)$, i.e.,
$$
\sigma( \x) = \lk i : x_i > 0\rk .
$$
The set of  nonnegative \corr{$n\times p$} matrices is denoted by $\R^{n\times p}_+$. \copa{A matrix $A\in \R^{n\times p}_+$ is called {\em positive} if   $\min\limits_{i,j}A_{ij}>0$. (Note that a completely positive matrix is not necessarily positive, since it may have zero entries.)
A matrix $A\in \R^{n\times n}$ is {\em diagonally dominant} if $|A_{ii}|\ge \sum\limits_{j\ne i} |A_{ij}|$ for every $1\le i\le n$, and is
{\em strictly diagonally dominant} if all these $n$ inequalities are strict.
For   two square matrices $A,B$} \copa{we denote}
$$A\oplus B =\left[\bea{cc} A & 0 \\ 0\T &B\ea\right]\, , $$
where $0$ is a suitable (possibly rectangular) zero matrix.

By $\csn n $ we denote the space of real symmetric $n\times n$ matrices, and by $\csnp$  the cone of  positive-semidefinite matrices,
$\csnp=\{X\in\csn n  :  X \succeq 0\}$. \copa{The cone of  nonnegative matrices in $\csn n$ is denoted by $\NN _n$, i.e.,}
\corr{$\NN_n=\csn n \cap \R^{n\times n}_+$.} The scalar product of two
matrices $U,V$ of same  order  is
$\langle U,V\rangle := \tr(U\T V) = \sum_{i,j} U_{ij} V_{ij}$. If $V=[\v_1, \ldots ,\v_p]\in \R_+^{n\times p}$, then  the factorization $M=VV^\top$ is equivalent to $M=\sum_{i=1}^p \v_i\v_i^\top$.
We refer to this sum as a {\em cp decomposition}.  When $p=\cpr M$ we say that the cp decomposition is \emph{minimal} (the cp factorization is minimal).

By $\inn {K}$ we denote the {\em \ignore{relative}interior} of a \ignore{convex}   set ${K}\subseteq \csn n$, $\bd K$ is the {\em boundary} of that set.
For a convex cone $K$, $\ext K$ denotes the set of all elements in $K$ which generate extreme rays of $K$.

Both the copositive cone $\cc n$ and the completely positive cone $\cs n$,
are pointed closed convex cones with nonempty interior.
As mentioned above, the copositive cone $\cc n$ and, in particular, its extremal rays, are important for the study of the cp-rank as any matrix on the boundary $\bd \cs n$ of $\cs n$ is orthogonal to an extremal ray of $\cc n$.
However, characterization of the extremal rays of $\cc n$ for $n> 5$ is itself a major open problem
in the study of $\cc n$. The explicit characterization of extremal rays of $\cc 5$ was completed by Hildebrand \cite{Hild11} only recently, and this work was essential for the arguments in~\cite{Shak13}.  One extremal ray of $\cc 5$ is generated by the so-called \emph{Horn matrix}
\beq{Horn} H=\left[\begin{array}{rrrrr}1&-1&1&1&-1\\-1&1&-1&1&1\\1&-1&1&-1&1\\1&1&-1&1&-1\\-1&1&1&-1&1\end{array}\right]\, ,\eeq
which historically was the first copositive matrix detected outside of $\PP_n+\NN_n$ (here $n=5$), see~\cite{Dian62}; attribution to Alfred Horn can be found in~\cite{Hall63}. In the sequel, a matrix $A$ is said to be {\em in the orbit} of a matrix $B$ if
$A= DP\T BPD$, where
$D$ is a positive-definite diagonal matrix and $P$ a permutation matrix. The {\em Horn orbit} consists of all matrices in the orbit of $H$; obviously, each matrix in the Horn orbit also generates an extremal ray of $\cc 5$. Finally, we address any extremal matrix in $\cc 5$ which is neither  in the Horn orbit  nor  in $\PP_5\cup \NN_5$ as a {\em Hildebrand matrix}; see description in~\cite{Hild11}.

\section{Orthogonality and diagonal dominance results}\label{orthdiagdom}
In this section we consider copositive and completely positive matrices
which are orthogonal to each other.
\cobf{The following theorem will be used in this section
to point out a property of matrices on the boundary of the copositive cone,
and also later to reduce the upper bound for $p_n$.}

\begin{thm}\label{orth}
Let $M\in \cs n$ be orthogonal to $A\in \cc n$, and let $M=\sum_{j=1}^p\x_j\x_j\T$ be any cp \cog{decomposition} of $M$.
Then
\begin{itemize}
\item[{\rm(a)}] For every $1\le i\le n$ the $i$-th column of $M$ is orthogonal to the $i$-th column of $A$.
\item[{\rm(b)}] If $i\in \supp {\x_j}$ for every $1\le j\le p$, then the $i$-th column of $A$ is in the nullspace of $M$.
\end{itemize}
\end{thm}

\begin{proof}
The scalar product of the $\ell$-th column of $M$ and the $i$-th column of $A$ is
\beq{MiAi} \e_\ell\T MA\e_i =\sum_{j=1}^p  \e_\ell\T\x_j \x_j\T A\e_i=\sum_{j:\ell\in \supp {\x_j}}\left(\e_\ell\T\x_j\right)\left [A\x_j\right]_i\, .
\eeq
(a) If $\ell=i$, the right-hand side in (\ref{MiAi}) is
$$\sum_{j:i\in \supp {\x_j}} \left(\e_i\T\x_j\right)\left [A\x_j\right]_i\, .$$
Since each \cog{$\x_j$ is  in $\R^n_+$} and satisfies $\qf{\x_j}A=0$ we have $\left[A\x_j\right]_k =0$ for all
$k\in \supp {\x_j}$~\cite[Rem.\,3.2, (3.1)]{Shak13}. In particular  $\left[A\x_j\right]_i=0 $ for all $j$'s in the sum above. Hence
$\e_i\T MA\e_i  =0$.\newline
(b) Let $\z=A\e_i$.
Suppose $i\in \supp {\x_j}$ for all $j\in\lk 1,\ldots ,p\rk$, then as above $\left[A\x_j\right]_i=0 $ for every $j$, and thus by~\reff{MiAi}
\[\left[M\z\right]_\ell=\e_\ell\T MA\e_i = \sum_{j=1}^p  \left (\e_\ell\T\x_j\right)\left[A\x_j\right]_i=\,   0\]
for every $\ell \in\lk 1, \ldots , n\rk $.
\end{proof}

\begin{remark} If $M=\x\x\T\in \cs n$ is of  rank one, both conditions~(a) and (b) above are in fact equivalent. Indeed, (a) says that the diagonal entries of $AM = (A\x)\x\T $, i.e.,
$x_i [A\x]_i = 0$ for all i. In other words, if $x_i>0$, then $[A\x]_i=0$. Whereas (b) states that for all $x_i>0$, we have
$[A\x]_i \x = \x\x\T A \e_i = M(A\e_i) = \o$, which, as $\x\neq \o$, is equivalent, again, to $[A\x]_i=0$.
\end{remark}

{Before we proceed, we note an interesting implication about {\em copositive} matrices on the boundary $\bd \cc n$. It is well known, and obvious by a \cobf{Gershgorin-type}
argument, that singular matrices (e.g. those on $\bd {\mathcal P}_n$) cannot be {strictly} diagonally dominant.} \ignore{Since symmetric strictly diagonally dominant matrices are positive-definite, nonsingular matrices in $\bd {\mathcal P}_n$ are also not strictly diagonally dominant.} \cobf{For matrices on $\bd \cc n$ we show that some form of ``anti- diagonal dominance'' can be established}:

\begin{corollary}
Let $A\in \bd \cc n$.  If $A\perp M\in \cs n$ and $M$ has no zero rows, then $A$ is in the orbit of some $\bar A\in \bd \cc n$ which satisfies $$\bar A_{ii} \le \sum_{j\neq i} | \bar A_{ij}|\quad \mbox{for all }i\, .$$
\end{corollary}
\bep {If $A=0$, this is trivial. Suppose  $0\neq A\in \bd \cc n$.  Since the completely positive matrix $M$ has no zero row by assumption, all its diagonal elements are positive. Thus we may scale $M$ by a positive definite diagonal matrix $D$  so that $\overline  M=DMD$ has $\diag \overline  M = \e$. Then
$\overline  A=D^{-1}AD^{-1}$ satisfies $\overline  A \perp \overline  M$. We may
 choose $\u_i \in \R^p_+$ where $p=\cpr \overline  M$, $1\le i \le n$, such that $\overline  M_{ij} = \u_i\T \u_j$ (in other words, $\overline  M$ is the Gram matrix of   $\u_1, \ldots , \u_n$).   For every $i$ we get from~Theorem~\ref{orth}(a) and  $\norm {\u_i}=1$ that
$$ \sum_{j\neq i}\overline  A_{ij} \u_i\T \u_j = (\overline  A\e_i)\T (\overline  M\e_i) - \overline  A_{ii}  =  -\overline  A_{ii} \, . $$
Passing to absolute values and applying Cauchy-Schwarz as well as the triangle inequality, we get
$$\overline  A_{ii} \le  \sum_{j\neq i} |\overline  A_{ij}| $$
for all $i$, the claimed assertion.}\ep

According to a result by Kaykobad~\cite{Kayk87}, any symmetric diagonally dominant matrix in ${\mathcal N}_n$ is already completely positive.
Motivated by this result, we could ask whether indeed these matrices are in the interior of $\cs n$. The answer is negative, a certificate being $I_n$\cogg{:
matrices in the interior of $\cs n$ are necessarily positive. In general, being nonsingular and positive  is  not a sufficient condition for an $n\times n$  completely positive matrix $A$ to be in $\inn{{ \cs n }} $, for instance, for some $A\perp H$ in~\reff{Horn}.} However, we can prove the following:

\begin{thm}\label{intddom}
\cogg{Let $M$ be an $n\times n$ symmetric matrix, $n\ge  3 $. If $M$ is  diagonally dominant and
 positive, then $M\in\inn{{ \cs n }} $.}
\end{thm}
\bep Let $\e= [1,\ldots , 1]\T\in \R^n$, let $J_n= \e\e\T$ denote the all ones $n\times n$ matrix, and let  $\mu :=\min_{i,j}M_{ij} >0$. Consider $M':=M-\mu J_n$. Since $n\ge 3$, $M'\in \NN_n$  is  strictly  diagonally dominant and therefore completely positive and nonsingular.  So we can put $M=VV\T$ where $\v_1 =\sqrt{\mu}\e$, and the remaining $\v_i$  come from the cp  factorization of $M'$. By Dickinson's characterization~\cite{Dick10} of $\inn{{  \cs n  }} $, the assertion is proved.\ep

Note that for $n=2$ there exist positive diagonally dominant matrices that are not in $\inn{{  \cs 2  }} $, e.g., $J_2$, which is singular, and
therefore on $\bd \cs 2$.

\section{Extreme copositive matrices of low rank}\label{extlow}

\ignore{As already noted,} If $A\in (\ext \cc n )\cap \NN_n$, then there is at most one
positive entry on or above the diagonal. If this entry is on the diagonal, we have $\rank A = 1$ and $A\in \PP_n$. If the 
positive entry is off the diagonal, then $A$ is in the orbit of the matrix $E_{12} = \e_1\e_2\T + \e_2\e_1\T$, \cor{and hence is of rank two. Next we will sharpen these assertions, basically dropping the nonnegativity assumption on $A$.} \cog{We will need the following  auxiliary result, on the role of zero entries on the diagonal of an extreme copositive matrix:}

\begin{lemma}\label{zeroext}
Suppose that $A\in (\ext\cc n)\setminus \NN_n$ can be decomposed as
$$A= \left[\bea{ccl} S &R \\ R\T &Q\ea\right ]\, , \quad \mbox{with }S\in \csn k\mbox{ and }\diag Q = \o\in \R^{n-k}\,
\cobf{, k\ge 1}.$$
\cobf{Then $R$ and $Q$ are zero matrices (of suitable orders) and $0\neq S\in \ext (\cc k)$.}
\end{lemma}
\bep
Since $\diag Q=\o$ and \cogg{$A\in \cc n$}, 
we deduce $Q\in \R_+^{(n-k)\times (n-k)}$ and $R\in \R_+^{k\times (n-k)}$.
Further,  \cogg{since $A\notin \NN_n$,} \cop{$S\in \R^{k\times k}$} has  at least one \cogg{negative}
element. Thus $k\ge 1$ and $S\neq 0$.
    We conclude that
    $$A= \left[\bea{cc} S &0 \\ 0 &0\ea\right ] + \left[\bea{cc} 0 &R \\ R\T &Q\ea\right ]\, ,$$
   where the rightmost matrix has no negative entries and therefore is copositive.  As $S\neq 0$, extremality of $A$ implies  that both $Q$ and $R$ have to be zero matrices, and $A=  S\oplus 0$ as well as extremality of $S$ in $\cc k$ follows.
\ep

We can now prove:

\begin{thm}\label{locoprank}
Let $A\in \ext\cc n$. Then
\begin{itemize}
\item[{\rm(a)}] $\rank A=1$ if and only if $A$ is positive-semidefinite.
\item[{\rm(b)}] $\rank A=2$ if and only if $A$ is
{in the orbit of $E_{12}=\e_1\e_2\T+\e_2 \e_1\T$.}
\end{itemize}
\end{thm}

\begin{proof}
The if parts are obvious. For the only if:
\begin{itemize}
\item[{\rm(a)}] If $\rank A=1$ then, since $A$ is symmetric, $A=\pm\x\x\T$ for some $\x\in \R ^n$. Since the diagonal entries of $A$ are nonnegative, $A=\x\x\T$.
\item[{\rm(b)}] {Suppose $\rank A=2$; if $A\in\NN_n$, then the result follows directly. So suppose $A$ has a negative entry. By extremality   $A\notin \PP_n$. Hence $A$ must be indefinite, i.e.~of the form $A= \u\u\T - \v\v\T =\x\y\T+\y\x\T$ (take, e.g., $\x= \frac 12 (\u+\v)$ and $\y= \u - \v$).
    For any $\z\in \R^n_+$ we have
    $$0\le \qf \z A = 2(\x\T\z)(\y\T \z)\, ,$$
    in particular \cob{$x_iy_i  \ge 0$} for all $i$.
        Put $\Omega:=\lk i: x_iy_i \neq 0\rk$. Then $x_i y_i >0$ for all $i\in \Omega$. {By permuting rows and columns if necessary, we may assume that $\Omega=\{1, \ldots, k\}$,} and $A$ can be decomposed
        as in Lemma~\ref{zeroext}, yielding $A= S\oplus 0$. So
    we may without loss of generality assume that $x_iy_i>0$ holds for all $i$, by investigating $S$ instead of $A$. By diagonal scaling we may now further assume that $x_iy_i=1$ for all $i$. Now,
    if either $\lk \x,\y\rk \subset \R^n_+$ or $\lk-\x,-\y\rk \subset \R^n_+$,
    we again would arrive at $A\in {\mathcal N}_n$. So we are done if we reduce the assumption $x_i>0>x_j$ {\em ad absurdum.} To this end, consider the $2\times 2$ block corresponding to these two indices $\lk i,j\rk$, putting $t=\frac {x_i}{x_j}<0$:
    $$\left[\bea{cc} 2x_iy_i &x_iy_j+ y_ix_j  \\ x_iy_j+ y_ix_j &2x_jy_j\ea\right ] = \left[\bea{cc} 2 &t+ \frac 1t  \\ t+ \frac 1t  &2\ea\right ]\in \cc 2\, .$$
    Copositivity of this $2\times 2$ matrix is equivalent to the condition $t+ \frac 1t \ge - 2$, which upon multiplication with $t<0$ amounts to $(t+1)^2 \le 0$ or $t=-1$. So all positive entries of $\x$ are equal, say $\alpha$, and all negative entries of $\x$ equal $-\alpha$. Hence $\y$ is a multiple of $\x$ and $\rank A = 1<2$, a contradiction.}
 \end{itemize}
  \end{proof}

\cogg{\begin{remark}\label{rnkextco}
Theorem \ref{locoprank} implies that for $n\le 4$ each matrix in $(\ext \cc n)\setminus\lk 0\rk $ has  rank 1 or 2.
The characterization of the extreme copositive $5\times 5$ matrices implies that each matrix in $(\ext \cc 5)\setminus\lk 0\rk$ has rank 1,2, or 5 \cite{Hild11}. What are the possible ranks of  matrices in $\ext \cc n$,  $n\ge 6$? Note that if there exists  $A\in \ext \cc n$ of rank $k$, then there exist also matrices in $\ext \cc {n+1}$  of rank $k$ ($A\oplus 0$, to name one). \end{remark}}

\cob{We proceed with} an immediate consequence for   positive and nonsingular matrices $M\in \bd \cs n$:

\begin{corollary}\label{rkthree}
{Suppose $M\in (\bd \cs n) \setminus \left (\bd {\PP}_n \cup \bd {\NN}_n\right )$ and
$0\neq A\in  \ext \cc n$. Then $A\perp M$ implies $\rank A\ge 3$.}
\cob{Moreover, no principal submatrix $S$ of such $A$ can be in the orbit of $E_{12}$.}
\end{corollary}
\bep
\cob{Because $M$ is assumed to be   positive, we know $A\notin \NN_n$; similarly, since $M$ is nonsingular, we conclude
$A\notin \PP_n$. Therefore
Theorem~\ref{locoprank} implies $\rank A \ge 3$. Next suppose a principal submatrix $S\neq 0$ of $A$ is in the orbit of $E_{12}$ and thus has $\diag S = \o$.} Then $A$ can be decomposed into
$$A= \left[\bea{cc} S &0 \\ 0 &0\ea\right ] + \left[\bea{cc} 0 &R \\ R\T &Q\ea\right ]\, ,$$
where $R$ has no negative entries and $Q$ is copositive. Hence the rightmost matrix is copositive, and therefore (by extremality of $A$
and $S\neq 0$)  must be the zero matrix. It follows $\rank A = \cor{\rank S\oplus 0 = \rank S = 2}$, but then Theorem~\ref{locoprank}(b) yields the contradiction $\cog{A}\in \bd {\mathcal N}_n$.
\ep

 \cor{\section{Matrices of fixed cp rank and order}}\label{genercpr}
We now turn to the study of the cp-rank of matrices in $\cs n$. 
In this section we consider the location of matrices with a certain fixed cp-rank in the cone $\cs n$,
and whether they constitute a substantial part of this cone.

First we observe that every possible cp-rank is attained at some matrix on the boundary:

\begin{proposition}\label{bd-cpr}
\cog{For every $1\le k\le p_n$ there exists a matrix $M_k \in \bd {\cs n}$ such that $\cpr M_k=k$.}
\end{proposition}

\bep
\cog{By \cite[Thm.\,3.4]{Shak13} there exists a matrix $M\in \bd  \cs n $ with $\cpr M=p_n$. Let
$M=\sum_{j=1}^{p_n} \v_j \v_j\T$ be a minimal cp-\cog{decomposition} of $M$, and let $M_k=\sum_{j=1}^{k} \v_j \v_j\T$.
Then $\cpr M_k\le k$, and strict inequality is impossible, because it would contradict the minimality
of the cp-\cog{decomposition} of $M$. That is, $\cpr M_k=k$. \cogg{Since $M$ is  on the boundary of  $\cs n$, there exists $A\in (\bd \cc n)$ with $A\neq 0$ such that $M$ is orthogonal to $A$. Then $\v_j\v_j\T \corr{\perp} A$ for every $j$, and thus $M_k\perp A$, and therefore \corr{$M_k\in \bd \cs n$.}}}
\ep

However, it is interesting to find out whether there are also interior matrices having a prescribed cp-rank, and whether they form a significant portion of the interior.  \cogg{For this purpose, we denote the set of completely positive matrices of order $n$ with cp-rank exactly equal to $k$ by
$$\cs {n,k} = \lk M \in \cs n : \cpr M = k\rk\, .$$}
\ignore{the ``percentage'' of matrices with a given cp-rank in $\cs n$.
To pose this question more precisely let the conic measure $\mu_c$ of a
measurable cone $\ckk$ in a Euclidean space be defined as
$$
\mu_c(\ckk):=\frac{\mu(\ckk\cap \BB)}{\mu( \BB)}
$$
where $\mu$ is the standard Lebesgue measure and $\BB$ is the unit ball centered at the origin.
\cor{Denote the set of completely positive matrices with cp-rank exactly equal to $k$ by
$$\cs {n,k} = \lk M \in \cs n : \cpr M = k\rk\, .$$
We know~\cite[Prop.2.4]{Shak13} that \cog{$\ca _{n,k} = \lk M\in \cs n : \cpr M \le k\rk$} is closed for all $k\in \lk 1,\ldots, p_n\rk$, so that
$\cs {n,k} = \ca _{n,k}\setminus \ca _{n,k-1}$ is a Borel set.}  \cog{The question mentioned above is:
\[\text{Compute, or estimate, $\mu_c(\ckk)$ for every $1\le k\le p_n$.} \]  We will now prove that}
$$
\mu_c(\cs {n,k}) > 0\quad  \Longleftrightarrow\quad  n\le k\le p_n\, .
$$}
\cor{The extreme case $k=p_n$ is easy: as shown in \cite[Cor.\,2.5]{Shak13}, the set $\cs {n,p_n}$ contains an open set, and thus,
$\inn { \cs {n,p_n} } \neq \emptyset $.
To prove this for all other $k$, we need a result which may also be of independent interest.} \corr{Beforehand}
\cog{note that every  $M\in \inn { \cs {n,p_n} }$ has a factorization $M =VV\T$ where $V\ge 0$ has $p_n$ columns, and  by one of Dickinson's characterizations of $\inn { \cs n }$  \cite[Thm.\,3.8]{Dick10}, there exists a factorization $M =WW\T$ where $W$ is 
positive (and has rank $n$). However, this does not necessarily imply that there is a  factorization $M =VV\T$ where $V\ge 0$ has $p_n$ columns \emph{and} all of these columns are
positive.}

\begin{proposition}\label{pos-cpr}
For every $n$ there exists a matrix $M\in \inn { \cs {n,p_n} }$ which has a minimal cp-factorization $M=VV^T$ with positive $V\in \R^{n\times p_n}$.
\ignore{There is always a matrix $M\in \inn { \cs {n,p_n} }$ such that there exists a   positive $n\times p_n$ matrix $V$ with $M =VV\T$.}
\end{proposition}
\bep
Let $M_0$ be some matrix in the interior of $\cs {n,p_n}$.
As in \cite{Shak13}, let $\v>\o$ be its Perron-Frobenius eigenvector to the eigenvalue $\lambda >0$.
If $M_0=V_0V_0\T$ where $V_0$ is a nonnegative $n\times p_n$-matrix, then no column
\cor{$V_0\e_i$ of $V_0$ is zero. Therefore $(V_0\T\v)_i = \v\T V_0\e_i >0$ for all $i$, in other words, \cop{$\tilde\x:=V_0\T\v >\o$}.}
From $\lambda\v=M_0\v=V_0V_0\T\v=V_0\tilde\x$ it follows that $\x:=\tilde\x/\lambda$ is a   positive vector with $\v=V_0\x$.
The choice of $M_0$ implies that for small $\eps >0$ the matrix $M=M_0+\eps \v\v\T$  also has $\cpr M =p_n$.
Now
\beq{MVrel}
M=M_0+\eps \v\v\T=V_0V_0\T+(V_0\x)\eps(V_0\x)\T=V_0(I_n+\eps \x\x\T)V_0\T\, .
\eeq
\cor{For $\delta =( \sqrt{1+\eps \x\T\x}-1)/\norm\x ^2>0$, define $C=I_n + \delta \x\x\T$. Then $C^2 = I_n + \eps\x\x\T$ and \cop{$V=V_0C = V_0 + \delta(V_0\x)\x\T \in \R^{n\times p_n}_+$} is   positive (since $\x >\o$ and \cop{$V_0\x  >\o$}). By~\reff{MVrel} we obtain $ VV\T = V_0 C^2 V_0\T = M$.} \ep

\begin{thm}\label{gener-cpr}
$$
\inn { \cs {n,k} }\neq \emptyset \quad  \Longleftrightarrow\quad  n\le k\le p_n\, .
$$
\end{thm}
\begin{proof}
For $k<n$ it follows from $\cpr M \ge \rank M $ that $\cs {n,k}$ is contained in the set of matrices with
rank at most $k$ and thus its interior is empty. \cor{We now show that $\inn { \cs {n,k} }\neq \emptyset$ if $n\le k\le p_n$.
To this end, Proposition~\ref{pos-cpr} ensures we can select a matrix $M=VV\T\in\inn { \cs {n,p_n} }$ with a   positive
$n\times p_n$ matrix $V= [\v_1, \ldots ,\v_{p_n}]$. As $M\in \inn { \cs n } \subset \inn  { \csnp  } $, we have $\rank V = n$ and} without loss of generality, let the first $n$ columns $\lk \v_1,\ldots , \v_n\rk$ of $V$ be linearly independent.
Now, let any $k$ with $n\le k\le p_n$ be given and consider the matrix
$$\overline M :=\sum_{j=1}^k \v_j\v_j\T\, .$$
\cor{Obviously $\cpr \overline M \le k$. On the other hand, $\cpr \overline M <k$ would contradict the \cog{minimality of the factorization} $M=VV\T$, so $\cpr \overline M = k$.} Let $\overline V := [\v_1,\ldots,\v_n]$ and   $\tilde V :=[\v_{n+1},\ldots , \v_k]$. Then $\overline V$ is   positive and nonsingular square, so by~\cite[Thm.\,2.3]{Duer08}, we have $\overline M = \overline V\,\overline V\T + \tilde V \tilde V\T \in\inn { \cs n }$.
Next consider the singular value decomposition of $\overline V=U_1\Sigma U_2$ with suitable \cor{orthonormal $n\times n$}
matrices $U_1$ and $U_2$  and a positive-definite diagonal  $n\times n$ matrix $\Sigma$.
Let $U_2\T \csn n U_1\T$ be the \cor{set of all matrices of order $n$ which result
from premultiplying  a symmetric matrix $Z$ by $U_2\T$ and postmultiplying it by $U_1\T$.}
Consider the map $\FF:  U_2\T \csn n U_1\T  \to\csn n$ defined by
$\FF(\Delta V):=(\overline V+\Delta V)(\overline V+\Delta V)\T$.
The derivative of $\FF$ at $\Delta V=0$ is given by the Lyapunov operator
$$\LL_{\overline V}:U_2\T \csn n U_1\T\to\csn n \quad\mbox{with}\quad \LL_{\overline V}(\Delta V)=(\Delta V)\overline V\T+\overline V(\Delta V)\T \, .$$
Given a symmetric right hand side $R$, solving $\LL_{\overline V}(U_2\T Z U_1\T)=R$ for a symmetric matrix $Z$
is equivalent to
\begin{eqnarray*}
U_2\T Z U_1\T U_1\Sigma U_2+ U_2\T\Sigma U_1\T  U_1 Z U_2&=& R\\
\Longleftrightarrow \quad Z\Sigma+\Sigma Z = U_2 R U_2\T.
\end{eqnarray*}
Evidently, this is uniquely solvable for a symmetric $Z$ so that by the inverse function theorem,
$\FF$ is invertible in an open neighborhood of $\Delta V = 0$, and the inverse function
satisfies $\overline V+\Delta V>0$ in this neighborhood, \cor{by continuity}.
Summarizing, for any (symmetric) matrix $\widehat  M$ in an open neighborhood of $\overline M$ there exists a   positive $n\times n$ perturbation matrix
$$\widehat  V = \overline V + \FF^{-1} (\widehat  M - \tilde V\tilde V\T)$$
of $\overline V$, such that $\widehat  M=\widehat  V\widehat  V\T+ \tilde V\tilde V\T = \widehat  V\widehat  V\T+  \sum_{j=n+1}^k \v_j\v_j\T\in  \cs n $,
\cor{which establishes $\cpr \widehat  M \le k$. But we know from~\cite[Cor~2.5]{Shak13} that all matrices $\widehat  M \in \cs n$ which are sufficiently close to $\overline M$ have  $\cpr \widehat  M \ge k$, so we conclude
 $\widehat  M\in \cs {n,k}$, hence $\overline M$ is an inner point of $\cs {n,k}$, and the results follow.}
\end{proof}

\section{New bounds for the cp-rank}\label{newcprbounds}
\cog{In this section we prove that the known upper bound $b_n$ on the cp-rank of $n\times n$ matrices can be reduced, for every $n\ge 6$. For $n=6$ we reduce the bound further in the next section.}
 First, we combine the idea of \cite{Shak09} with Theorem~\ref{orth} to show that $p_n$ is  strictly less than $b_n$ for every $n\ge 3$.

\begin{thm}\label{bn-k+1}
For $n\ge 2$, if $A\in \bd \cc n$ has $k\ge 2$   positive diagonal elements, and $M\in \cs n$ is orthogonal to $A$,
then $\cpr M\le b_n - k+1$.
\end{thm}

\begin{proof}
We may assume that $A_{ii}>0$ for $i\in \lk 1, \ldots, k\rk $. Let
\[{\cal L }=\left\{ B\in \csn n  : \e _i\T B A\e _i=0 ~ ~\mbox{for all }i\in \lk 1, \ldots, k\rk \right\}\, .\]
Then $\{\v\v\T : \v\in \R^n_+ 
\text{ and } \v\T A\v=0 \}\subseteq \cal{L}$ by Theorem~\ref{orth}. The subspace $\cal L$ is isomorphic to the solution
space of the homogenous system of $k$ equations in variables $b_{ij}$, $1\le i\le j\le n$,
$$A_{ii}b_{ii}+\sum_{j<i}A_{ij} {\cog{b_{ji}}}+\sum_{i<j} A_{ij}{\cog{b_{ij}}}=0\, , \quad i\in \lk 1, \ldots, k\rk \, .$$
Since the diagonal matrix with
$A_{ii}$, $i=1, \ldots, k$, on the diagonal is a submatrix of the coefficients matrix, the rank of the coefficients matrix is $k$. Thus
$\dim {\cal L }={{n+1}\choose 2}-k$. Next suppose $M\in \cs n$ is orthogonal to $A\in \bd \cc n$. Then $M\in \conv\{\v\v\T : \v\in \R^n_+ 
\text{ and } \v\T A\v=0 \}$ which is a convex cone contained in $\LL$, and by Caratheodory's theorem
$\cpr M\le \dim {\cal L} ={{n+1}\choose 2}-k$.
\end{proof}

Thus for certain completely positive matrices on $\bd  \cs n$ we get the following bound on the cp-rank:

\begin{corollary}\label{bn-4}
For $n\ge 5$, if $A\in (\bd \cc n)\setminus \cob{\NN_n}$ , and $M\in \cs n$ is orthogonal to $A$, then $\cpr M\le b_n - 4$.
\end{corollary}

\begin{proof}
We may assume $A\in \ext \cc n$.  If $A$ is positive-semidefinite it follows
from orthogonality and $M\succeq 0$ that $\rank(M)\le n-1$ and
thus, by~\reff{babe}, $\cpr(M)\le b_{n-1}\le b_n-4$. We now assume that $A$ is \corr{indefinite}.
Let $k$ be the number of positive diagonal elements of $A$. If $k=0$, we would get $A\in \NN_n$ which we ruled out by assumption.
   If $1\le k\le 4$, then by Lemma~\ref{zeroext}, we get, up to permutations of rows and columns, $A={  S}\oplus 0$  where $ S$ is  a copositive matrix of order  $k$.
  So $S\in \PP_k+\NN_k$, and therefore
$A\in \PP_n+\NN_n$. But then, by extremality, $A$ is  either positive-semidefinite or nonnegative, in contradiction to our assumptions. Thus $k\ge 5$,
 and by Theorem \ref{bn-k+1} we get $\cpr M\le b_n-4$.
\end{proof}

It is beneficial to introduce a cp-rank bound for   positive matrices on the boundary of $\cs n$: \cog{Let
$p_n^* : = \max\lk \cpr M : M\in \bd \cs n\, , \; \min\limits_{i,j} M_{ij} > 0\rk$.}

\begin{thm}\label{pnbybd}
For $n\ge 2$, there exists a   positive matrix $M\in (\bd \cs n)\setminus (\bd\NN_n) $ such that
\[p_n\le \cpr M + 1\, .\]
\cor{Hence we get
\beq{pnstar} p_n^* \le p_n \le p_n^* +1\, ;\eeq the right inequality is an equality for \cob{$n\in\lk 2, 3\rk$} whereas the left inequality is an equality for  $n\in\lk 4, 5\rk$.}
\end{thm}

\begin{proof}
Let $\overline  M\in  \inn{{ \cs n }}$ be a matrix such that $\cpr \overline  M=p_n$~\cite[Cor.\,2.5]{Shak13}. Let $\delta >0$ be such that $M=\overline  M-\delta \e_n\e_n\T\in \bd \cs n$. Clearly, $p_n=\cpr \overline  M\le \cpr M + 1$.
Since $M$ has
positive off-diagonal entries in the last row and it is positive-semidefinite, \cog{we have} $M_{nn}>0$ and thus $M$ is positive, \cor{so that $\cpr M \le p_n^*$. Hence  $p_n^* \le p_n \le p_n^* +1$.
The last assertions follow from $b_2 = 2=p_3-1$ and from the fact that there are singular   positive matrices $M\in\bd \cs n$ with $\cpr M =p_n$ for $n\in\lk 4, 5\rk$~\cite[Rem.\,2.1,Cor.\,4.1]{Shak13}.} \end{proof}

\begin{remark}
For $n\le 4$, \cor{the matrix} $M$ in Theorem \ref{pnbybd} is necessarily singular. Thus for $n=2$  we have $\rank M=1$ and $\cpr M=1$, which yields $p_2\le 2$, a bound which is tight. If \cor{$n\in \lk 3,4\rk$ then} $\rank M= n-1$, and thus $\cpr M\le b_{n-1}$, and $p_n\le {n\choose 2}$. For $n=3$ this gives $p_3\le 3$, which is also a tight bound. But for \ignore{$n\in \lk 4,5\rk$} \cog{$n=4$  this yields $p_4\le 6$, which is not tight. For $n=5$, it turns out that \cogg{though} $p_5=p_5^*$ is attained at a singular matrix \cite[Cor.\,4.1]{Shak13}, \corr{we have} $p_5=6<{5\choose 2} =10$.} Still, for $n\ge 6$ we get an improvement of the known bound $p_n\le b_n$, and for $n=6$ we will improve it further below.
\end{remark}

\cogg{By the above, for the first time we have a proof that $b_n$ is not a tight upper bound on the cp-rank of completely positive matrices of \corr{any} order $n \ge 5$. More precisely: }

\begin{corollary}\label{pnlebn-3}
For $n\ge 5$, \cor{we have}
$$\cor{p_n^* \le b_n-4\quad\mbox{and}\quad} p_n\le b_n-3\, .$$
\end{corollary}

\begin{proof}
Let $M\in(\bd  \cs n) \setminus (\bd \NN_n)$ be a   positive matrix with $\cpr M =p_n^*$.
\ignore{
If $M$ is singular, then $\cpr M\le b_{n-1}$ by~\reff{babe} and thus $\cpr M \le  b_n-4$ \cob{(since $n\ge 5$)}.
If $M$ is nonsingular, then it is orthogonal to a matrix $A\in \ext \cc n$, which has \cob{a negative entry}. }
This matrix \cob{$M$ is  orthogonal to a matrix $A\in \ext \cc n$. As $M$ is   positive, $A\not\in\NN_n$}.
By Corollary~\ref{bn-4}, we have $\cpr M\le b_n-4$, and $p_n\le b_n-3$ by Theorem~\ref{pnbybd}.
\end{proof}

\begin{remark}\label{posdiag}
For $n=6$ this bound is $p_6\le \cor{17}$, but may be slightly improved.
This case is studied in~Section~\ref{psix}.
\end{remark}

Beforehand we note a further result valid for arbitrary order.

\begin{thm}\label{zerobd}
If  $M\in \cs n$ has a zero entry, i.e, $M\in \bd {\mathcal N}_n$, then
$$\cpr M \le 2p_{n-1}\, .$$
\end{thm}
\bep \cob{We may and do suppose that $M_{1n}=0$.}
\ignore{Consider again the decomposition~\reff{zerodecomp}. , and this implies $M_+(V) =  {M'}\oplus 0$ for some $M'\in \cs{(n-1)}$, so that $\cpr M_+(V) = \cpr M' \le p_{n-1}$.}
Let $M=\sum_{i=1}^p \w_i\w_i\T$  be a cp-\cog{decomposition} of $M$. Define
 $\Omega_1 :=\{i\in \{1, \ldots, p\} : 1\in \supp {\w_i}\}$ and  $\Omega_2:=\{1, \ldots, p\}\setminus \Omega_1$, as well as $M_j:=\sum_{i\in \Omega_j}\w_i\w_i\T$, for $j=1,2$. Then $M_1= {M_1'}\oplus 0$ and $M_2=  0\oplus {M_2'}$, where $M_1', M_2'$ are matrices in $\cs {n-1}$.
The result follows from $\cpr M=\cpr (M_1+M_2) \le \cpr M_1 + \cpr M_2= \cpr M_1'+ \cpr M_2'\le 2p_{n-1}$.
 \ep

\cor{ \begin{remark}
Note that for all $n\ge 6$, we have $2d_{n-1} \le b_n - 3$, the bound from Corollary~\ref{pnlebn-3}. However, compared with the upper bound from Theorem~\ref{bn-k+1}, we have $2p_{n-1} \ge 2d_{n-1} > b_n - k +1$, whenever $k>\frac n2$, so that Theorem~\ref{zerobd} is interesting only for small $k$.
 \end{remark}}

\section{Cp-rank of matrices of order six}\label{psix}
\cog{In this section we improve the upper bound on the cp-rank of completely positive matrices
of order 6.
First we consider matrices in $\bd \cs 6$ which are orthogonal to  $ S\oplus 0$, where $S\in \ext\cc 5$ is either
 in the orbit of the $5\times 5$ Horn matrix $H$ or a Hildebrand matrix.} \ignore{Below, we use the symbol
 `` $\corr{\boxplus}$ '' to denote the sum modulo 5
 of two elements in $\{1, 2, 3, 4, 5\}$.} Below, the sum of two elements in $\{1, 2, 3, 4, 5\}$  is the sum modulo 5.

\begin{proposition}\label{hornarg}
Let $S$ be either in the orbit of the $5\times 5$ Horn matrix $H$ or a Hildebrand matrix. Suppose that
$M\in \cs {\cob{6}}$ is orthogonal to $ S\oplus 0$. Then $\cpr M \le 15$.
\end{proposition}
\bep
If $M=VV\T$, $V\ge 0$, then
each column of $V$ is a nonnegative linear combination of three vectors, $\e_6$, $\e_i+\e_{i+ 1}$ and
$\e_{i+ 1}+\e_{i+ 2}$, for some $1\le i\le 5$~{\cob{\cite[Thm.\,4.4]{Shak13}}}.
Let \[W=[ \e_1+\e_2|\e_2+\e_3|\e_3+\e_4|\e_4+\e_5|\e_5+\e_1|\e_6]\, .\] Then $V=WX$, where each column of $X$ has support of at most $3$ elements, contained in
a set of the form $\{i, i+ 1,6\}$ with $1\le i\le 5$. For each such $i$, let $X_i$ consist of the columns of $X$ whose support is
contained in $\{i, i+ 1,6\}$.  Then, again up to permutations of rows and columns, $X_i X_i\T = {Y_i}\oplus 0$ with $Y_i\in \cs 3$ so that
$$\cpr X_i X_i\T = \cpr Y_i \le p_3= 3\, .$$
Therefore $\cpr XX\T = \cpr \sum\limits_{i=1}^5 X_iX_i\T \le  \cogg{\sum\limits_{i=1}^5 } \cpr X_iX_i\T\le 15$.
\ep

\begin{thm}
We have $p_6\le 15$.
\end{thm}
\bep
By \cite[Thm.\,3.4]{Shak13}, we know that $p_6 = \cpr M$ for some $M\in (\bd \cs 6 )\setminus (\bd \PP_6)$. Moreover, if $M$ had a zero entry, we get from Theorem~\ref{zerobd} that $\cpr M \le 2p_5 = 12$. Suppose now that $M\in (\bd \cs 6) \setminus \left (\bd {\PP}_6 \cup \bd {\NN}_6\right )$. Then Corollary~\ref{rkthree} gives $\rank A \ge 3$ for all \cop{$A\in  (\ext \cc   6)\cap M^\perp$}, and at least one such $A\neq 0$ exists as $M\in\bd \cs 6$. 
\cog{Now either all diagonal elements of $A$ are 
positive,
in which case by Theorem \ref{bn-k+1} $\cpr M\le b_6-5=15$, or  $A$ has at least one zero on the diagonal. 
By Lemma~\ref{zeroext},} 
$A= S\oplus 0$ with $S\in \ext \cc 5$. Since  $\rank S= \rank A \ge 3$, we conclude that $S$ is either in the orbit of $H$ or a Hildebrand matrix. Then Proposition~\ref{hornarg} gives $\cpr M \le 15$, and the claim is proved.
\ep

\cogg{We thus cut the bracket for $p_6$ in about half, since $b_6=20$ and $d_6=9$.}
\corr{The same argument could be used also for $n\in\lk 7,8\rk$, but it would not further improve upon the bounds yielded already by the general improvement in Corollary~\ref{pnlebn-3}.}

\medskip

\bibliographystyle{plain}


\end{document}